\documentclass[
11pt,%
tightenlines,%
twoside,%
onecolumn,%
nofloats,%
nobibnotes,%
nofootinbib,%
superscriptaddress,%
noshowpacs,%
centertags]%
{revtex4}
\usepackage{ljm_withoutLJM}
\graphicspath{{./res/}}

\begin{document}

\titlerunning{Inverse problems for Dirac operators with constant delay }

\authorrunning{Buterin, Djuri\'c}

\vspace*{-5mm}

\title{Inverse problems for Dirac operators with constant delay: uniqueness, characterization, uniform stability}

\author{\firstname{Sergey}~\surname{Buterin}}
\email[E-mail: ]{buterinsa@info.sgu.ru}
\author{\firstname{Neboj\v{s}a}~\surname{Djuri\'c}}
\email[E-mail: ]{nebojsa.djuric@aggf.unibl.org}

\affiliation{ Department of Mathematics, Saratov State University}

\affiliation{ Faculty of Architecture, Civil Engineering and Geodesy, University of Banja Luka}

\begin{abstract}
We initiate studying inverse spectral problems for Dirac-type functional-differential operators with constant delay. For simplicity, we
restrict ourselves to the case when the delay parameter is not less than one half of the interval. For the considered case, however, we give
answers to the full range of questions usually raised in the inverse spectral theory. Specifically, reconstruction of two complex
$L_2$-potentials is studied from either complete spectra or subspectra of two boundary value problems with one common boundary condition. We
give conditions on the subspectra that are necessary and sufficient for the unique determination of the potentials. Moreover, necessary and
sufficient conditions for the solvability of both inverse problems are obtained. For the inverse problem involving the complete spectra, we
establish also uniform stability in each ball. For this purpose, we use recent results on uniform stability of sine-type functions with
asymptotically separated zeros.
\end{abstract}


\maketitle

\section{Introduction and main results}\label{sec1}

\medskip
In the last decade, there appeared a significant interest in inverse problems for Sturm--Liouville-type operators with delay $a\in(0,\pi)$
(see \cite{Pik91, FrYur12, Yang, Ign18, BondYur18-1, VPV19, DV19, ButYur19, Dur, VPVC20, DB21, DB21-2, BMSh21, DB22} and references therein):
\begin{equation}\label{1.0}
\ell y:= -y''(x)+q(x)y(x-a)=\lambda y(x), \quad 0<x<\pi.
\end{equation}
In particular, it is known that the complex square-integrable potential $q(x)$ vanishing on $(0,a)$ is uniquely determined by specifying the
spectra of two boundary value problems for equation (\ref{1.0}) with a common boundary condition at zero as soon as $a\in[2\pi/5,\pi).$ The
recent series of papers \cite{DB21, DB21-2, DB22} establishes, however, that it is never possible for $a\in(0,2\pi/5).$ For more details, see
the survey in \cite{BMSh21}. Classical methods for solving inverse problems for differential operators \cite{Mar, FY01, Lev, GasLev, GasDzab,
Mal99, MykPuy, Hry11-2, GorYur, SavShk10, Hry11-1, Mak} are not applicable for operators with delay as well as for other classes of nonlocal
operators. However, they are often more adequate for modelling various physical processes frequently possessing a nonlocal nature \cite{Mysh,
BellCook, Nor, Hal, Skub, Mur1,Mur2,WY}.

The present paper opens a new chapter of the inverse spectral theory for operators with constant delay by introducing and studying the
Dirac-type system of the form
\begin{equation}\label{1.1}
By'(x)+Q(x)y(x-a)=\lambda y(x), \quad 0<x<\pi,
\end{equation}
where
$$
    B = \begin{bmatrix} 0 & 1 \\ -1 & 0 \end{bmatrix}, \quad
    y(x) = \begin{bmatrix} y_1(x) \\ y_2(x) \end{bmatrix}, \quad
    Q(x) = \begin{bmatrix} q_{11}(x) & q_{12}(x) \\ q_{21}(x) & q_{22}(x) \end{bmatrix},
$$
while $q_{jk}(x),$ $j,k=1,2,$ are complex-valued functions in $L_2(0,\pi),$ and $Q(x)=0$ on $(0,a).$


For any fixed pair of $\nu,j\in\{1,2\},$ we denote by ${\cal B}_{\nu,j}(Q)$ the boundary value problem for equation~(\ref{1.1}) under the
boundary conditions
$$
y_\nu(0) = y_j(\pi) = 0.
$$

In the classical case $a=0,$ inverse problems for (\ref{1.1}) were studied in \cite{Lev, GasLev, GasDzab, Mal99, MykPuy, Hry11-2, GorYur,
Mak} and other works. In particular, two canonical forms can be distinguished (see, e.g., \cite{Lev}):
\begin{equation}\label{1.4}
q(x):=q_{11}(x)=-q_{22}(x), \quad p(x):=q_{12}(x)=q_{21}(x)
\end{equation}
and
\begin{equation}\label{1.4.1}
q_{12}(x)=q_{21}(x)=0.
\end{equation}
However, the forms (\ref{1.4}) and (\ref{1.4.1}) are not canonical anymore when $a>0$ because of the non-locality. Moreover, Remark~2 in the
next section reveals that the spectrum of any problem ${\cal B}_{\nu,j}(Q)$ may carry essential information only on the functions
$q_{11}(x)-q_{22}(x)$ and $q_{12}(x)+q_{21}(x).$ For this reason, we assume everywhere (except the first part of Section~2 before Remark~2)
that condition (\ref{1.4}) is fulfilled. Another crucial difference from the case $a=0$ is described below in Remark 1. We also note that
various aspects of inverse spectral problems for Dirac-type systems with {\it integral} delay were studied in \cite{BB17, Bon18-2,
BondBut-20}.

For simplicity, we restrict ourselves to the case $a\ge\pi/2,$ when the dependence of the characteristic functions of the problems ${\cal
B}_{\nu,j}(Q)$ on $Q(x)$ is linear (for more details, see the next section). The nonlinear case $a\in(0,\pi/2)$ requires a separate
investigation. But for the considered case we give answers to the full range of questions usually raised in the inverse spectral theory.

Mentioned Remark~2 also indicates that, from the informational point of view, it is sufficient to specify the spectra of the problems ${\cal
B}_{1,1}(Q)$ and ${\cal B}_{1,2}(Q).$ We begin with the following theorem.

\begin{theorem}\label{th1}
So let condition (\ref{1.4}) be fulfilled. For $j\in\{1,2\},$ the problem ${\cal B}_{1,j}(Q)$ has infinitely many eigenvalues
$\lambda_{n,j},$ $n\in{\mathbb Z},$ of the form
\begin{equation}\label{1.3}
\lambda_{n,j}=n+\frac{1-j}2+\varkappa_{n,j}, \quad \{\varkappa_{n,j}\}\in l_2.
\end{equation}
\end{theorem}

Consider first the following inverse problem.

\medskip
{\bf Inverse Problem 1.} Given the spectra $\{\lambda_{n,j}\}_{n\in\mathbb Z},\,j=1,2,$ find the functions $p(x)$ and $q(x).$

\medskip
As will be seen below, Inverse Problem 1 is overdetermined (as well as the analogous problem for the Sturm--Liouville operator with delay
\cite{ButYur19}). Specifically, the potentials $p(x)$ and $q(x)$ can be uniquely determined by appropriate subspectra. Therefore, we consider
also another statement of the inverse problem. For this purpose, we fix increasing sequences $\{n_{k,j}\}_{k\in{\mathbb Z}},$ $j=1,2,$ of
integer numbers.

\medskip
{\bf Inverse Problem 2.} Given the subspectra $\{\lambda_{n_{k,j},j}\}_{k\in\mathbb Z},\,j=1,2,$ find the functions $p(x)$ and $q(x).$

\medskip
The following theorem gives a necessary and sufficient condition for the uniqueness of its solution.

\begin{theorem}\label{th2}
The solution of Inverse Problem~2 is unique if and only if each of the systems $\{\exp(in_{k,j}x)\}_{k\in{\mathbb Z}},$ $j=1,2,$ is complete
in $L_2(a-\pi,\pi-a).$
\end{theorem}

The next theorem provides sufficient conditions for the solvability of Inverse Problem~2.

\begin{theorem}\label{th3}
Let $j\in\{1,2\},$ and let the system $\{\exp(in_{k,j}x)\}_{k\in{\mathbb Z}}$ be a Riesz basis in $L_2(a-\pi,\pi-a).$ Then for any complex
sequence $\{\mu_{k,j}\}_{k\in\mathbb Z}$ to be a subspectrum of the problem ${\cal B}_{1,j}(Q),$ it is sufficient to have the asymptotics
\begin{equation}\label{1.6}
\mu_{k,j}=n_{k,j}+\frac{1-j}2+\varkappa_{k,j}, \quad \{\varkappa_{k,j}\}\in l_2.
\end{equation}
\end{theorem}

When the required Riesz-basisness takes place automatically, Theorem~\ref{th3} gives a {\it necessary} and sufficient condition for the
solvability of Inverse Problem~2 in terms of sole asymptotics. For example, let us refer to $\{\lambda_{mk,j}\}_{k\in\mathbb Z}$ with some
natural $m$ as {\it $m$-th subspectrum} of the problem ${\cal B}_{1,j}(Q).$ Then the following theorem holds.

\begin{theorem}\label{th4}
Let $j\in\{1,2\}$ and $a=\pi-\pi/m,$ $m-1\in{\mathbb N}.$ Then for any complex sequence $\{\mu_{k,j}\}_{k\in\mathbb Z}$ to be an $m$-th
subspectrum of ${\cal B}_{1,j}(Q),$ it is necessary and sufficient to have the asymptotics
$$
\mu_{k,j}=mk+\frac{1-j}2+\varkappa_{k,j}, \quad \{\varkappa_{k,j}\}\in l_2.
$$
\end{theorem}

According to Theorem~\ref{th2}, for the unique solvability of the inverse problem, it is always sufficient to specify the so-called even
subspectra or, in our present terminology, $2$-nd subspectra $\{\lambda_{2k,1}\}_{k\in\mathbb Z}$ and $\{\lambda_{2k,2}\}_{k\in\mathbb Z}.$
Hence, Inverse Problem~1 is overdetermined. However, in spite of the non-minimality of its input data, one can formulate necessary and
sufficient conditions for its solvability analogously to how it was made for a scalar pencil with two delays in \cite{BMSh21}. Specifically,
the following theorem holds.

\begin{theorem}\label{th5}
For any sequences of complex numbers $\{\lambda_{n,1}\}_{n\in{\mathbb Z}}$ and $\{\lambda_{n,2}\}_{n\in{\mathbb Z}}$ to be the spectra of
some boundary value problems ${\cal B}_{1,1}(Q)$ and ${\cal B}_{1,2}(Q),$ respectively, it is necessary and sufficient to satisfy the
following two conditions:

(i) For $j=1,2,$ the sequence $\{\lambda_{n,j}\}_{n\in{\mathbb Z}}$ has the form~(\ref{1.3});

(ii) The exponential types of the functions $\Delta_1(\lambda)+\sin\lambda\pi$ and $\Delta_2(\lambda)-\cos\lambda\pi$ do not exceed $\pi-a,$
where the functions $\Delta_j(\lambda)$ are determined by the formulae
\begin{equation}\label{1.8}
\Delta_1(\lambda)=\pi(\lambda_{0,1}-\lambda)\prod_{|n|\in{\mathbb N}}\frac{\lambda_{n,1}-\lambda}n \exp\Big(\frac\lambda{n}\Big), \quad
\Delta_2(\lambda)=\prod_{n\in{\mathbb Z}}\frac{\lambda_{n,2}-\lambda}{n-1/2} \exp\Big(\frac\lambda{n-1/2}\Big).
\end{equation}
\end{theorem}

We note that the functions $\Delta_1(\lambda)$ and $\Delta_2(\lambda)$ constructed by the formulae in (\ref{1.8}) are the characteristic
functions of the problems ${\cal B}_{1,1}(Q)$ and ${\cal B}_{1,2}(Q),$ respectively (see the next section and Lemma~\ref{lem2}).

Finally, let us formulate the uniform stability for Inverse Problem~1. For this purpose, along with the problems ${\cal B}_{1,j}(Q),$ we will
consider problems ${\cal B}_{1,j}(\tilde Q)$ of the same forms but with a different potential matrix $\tilde Q(x)$ obeying the condition that
is analogous to condition (\ref{1.4}). As usual, if some symbol $\alpha$ denotes an object related to the problem ${\cal B}_{1,j}(Q),$ then
this symbol with tilde $\tilde\alpha$ will denote the analogous object related to the problem ${\cal B}_{1,j}(\tilde Q).$

\begin{theorem}\label{th6}
For any fixed $r>0,$ there exists $C_r>0$ such that the following estimate holds:
\begin{equation}\label{1.9}
\|q-\tilde q\|_{L_2(a,\pi)}+\|p-\tilde p\|_{L_2(a,\pi)}\le C_r\Big(\|\{\lambda_{n,1}-\tilde\lambda_{n,1}\}_{n\in{\mathbb Z}}\|_{l_2}
+\|\{\lambda_{n,2}-\tilde\lambda_{n,2}\}_{n\in{\mathbb Z}}\|_{l_2}\Big)
\end{equation}
as soon as $\|\{\lambda_{n,j}-n+(j-1)/2\}_{n\in{\mathbb Z}}\|_{l_2}\le r$ and $\|\{\tilde\lambda_{n,j}-n+(j-1)/2\}_{n\in{\mathbb
Z}}\|_{l_2}\le r$ for $j=0,1.$
\end{theorem}

We note that the uniform stability of inverse problems for the classical Sturm--Liouville and Dirac operators in the selfadjoint case was
studied in \cite{SavShk10,Hry11-1} and \cite{Hry11-2}, respectively. However, we use a different approach suggested in \cite{But21-1,But21-2}
for other nonlocal operators. An important part of that approach was the proof of the uniform stability for recovering the characteristic
function of the operator under consideration from its zeros, i.e. from the spectrum of this operator. Afterwards, results of this type were
extended to arbitrary sine-type functions with asymptotically separated zeros \cite{But22}, which include characteristic functions of many
concrete operators frequently appearing in applications. We obtain all required assertions for the characteristic functions
$\Delta_1(\lambda)$ and $\Delta_2(\lambda)$ including asymptotics of their zeros (and, thus, Theorem~\ref{th1}) as simple corollaries of the
corresponding general assertions in \cite{But22}.

\medskip
{\bf Remark 1.} In the above-mentioned classical cases \cite{SavShk10,Hry11-1,Hry11-2}, the uniform stability of the inverse problem was
studied only in the selfadjoint case, and it takes place only when there is a non-reducible positive gap between neighboring eigenvalues.
This restriction was explained in \cite{But21-1, But21-2}. It is connected with the fact that spectra of the problems ${\cal B}_{1,1}(Q)$ and
${\cal B}_{1,2}(Q)$ cannot intersect if $a=0.$ Moreover, in the selfadjoint case they interlace. That is why unlimited rapprochement of
neighboring eigenvalues is incompatible with the uniform stability of the inverse problem. However, the picture becomes completely different
for $a>0.$ In particular, Theorem~4 implies that both spectra may possess any finite numbers of common eigenvalues. Moreover, the
corresponding operators are non-selfadjoint and they may have multiple eigenvalues, which also does not affect the uniform stability.

\medskip
The paper is organized as follows. In the next section, we construct and study a fundamental matrix-solution of equation (\ref{1.1}),
introduce the characteristic functions of all problems ${\cal B}_{\nu,j}(Q)$ and prove Theorem~\ref{th1}. In Section~\ref{sec3}, we obtain
auxiliary assertions for related entire functions including their uniform stability. Proofs of Theorems~\ref{th2}--\ref{th6} are given in
Section~4.

\section{Fundamental solution and characteristic functions}\label{sec2}

\medskip
Allowing  $a\in(0,\pi),$ we denote by $Y(x,\lambda)$ the fundamental matrix-solution of equation (\ref{1.1}) such that
$$
 Y(x,\lambda)=\begin{bmatrix} y_{1,1}(x,\lambda) & y_{1,2}(x,\lambda) \\ y_{2,1}(x,\lambda) & y_{2,2}(x,\lambda) \end{bmatrix}, \quad
 Y(0,\lambda)=\begin{bmatrix} 1 & 0 \\ 0 & 1 \end{bmatrix}.
$$
Then, for $\nu,j\in\{1,2\},$ eigenvalues of the problem ${\cal B}_{\nu,j}(Q)$ coincide with zeros of the entire function
\begin{equation}\label{2.00}
\Delta_{\nu,j}(\lambda):=y_{j,3-\nu}(\pi,\lambda),
\end{equation}
which is called {\it characteristic function} of ${\cal B}_{\nu,j}(Q).$

When $Q(x)=0$ a.e. on $(0,\pi),$ we, obviously, have $Y(x,\lambda)=Y_0(x,\lambda),$ where
\begin{equation}\label{2.0}
Y_0(x,\lambda):=\begin{bmatrix} \cos\lambda x & -\sin\lambda x \\ \sin\lambda x & \cos\lambda x \end{bmatrix}.
\end{equation}
By Lagrange's method of variation of parameters, one can establish that the matrix-function $Y(x,\lambda)$ is a solution of the integral
equation
\begin{equation}\label{2.1}
Y(x,\lambda)=Y_0(x,\lambda)+B\int_a^xY_0(x-t,\lambda)Q(t)Y(t-a,\lambda)\,dt,
\end{equation}
which can be easily checked also by direct substitution into (\ref{1.1}). Solving equation (\ref{2.1}) by the method of successive
approximations we arrive at the following representation of $Y(x,\lambda)$ as a finite sum:
\begin{equation}\label{2.2}
Y(x,\lambda)=\sum_{k=0}^NY_k(x,\lambda), \quad Y_k(x,\lambda):=B\int_{ka}^xY_0(x-t,\lambda)Q(t)Y_{k-1}(t-a,\lambda)\,dt, \quad
k=\overline{1,N},
\end{equation}
where $N$ is such that $a\in[\pi/(N+1),\pi/N).$

In particular, in our considered case $a\ge\pi/2,$ formula (\ref{2.2}) takes the form
\begin{equation}\label{2.3}
Y(x,\lambda)=Y_0(x,\lambda)+B\int_a^xY_0(x-t,\lambda)Q(t)Y_0(t-a,\lambda)\,dt.
\end{equation}
Substituting (\ref{2.0}) into (\ref{2.3}) and taking the definition (\ref{2.00}) into account, one can calculate
\begin{equation}\label{2.4}
\Delta_{1,1}(\lambda)=-\sin\pi\lambda +V_1(\lambda) -\int_{a-\pi}^{\pi-a} (p_{11}-p_{22})(x)\cos\lambda x\,dx + \int_{a-\pi}^{\pi-a}
(p_{12}+p_{21})(x)\sin\lambda x\,dx,
\end{equation}
\begin{equation}\label{2.5}
\Delta_{1,2}(\lambda)=\cos\pi\lambda +V_2(\lambda) -\int_{a-\pi}^{\pi-a} (p_{11}-p_{22})(x)\sin\lambda x\,dx - \int_{a-\pi}^{\pi-a}
(p_{12}+p_{21})(x)\cos\lambda x\,dx,
\end{equation}
\begin{equation}\label{2.6}
\Delta_{2,1}(\lambda)=\cos\pi\lambda +V_2(\lambda) +\int_{a-\pi}^{\pi-a} (p_{11}-p_{22})(x)\sin\lambda x\,dx + \int_{a-\pi}^{\pi-a}
(p_{12}+p_{21})(x)\cos\lambda x\,dx,
\end{equation}
\begin{equation}\label{2.7}
\Delta_{2,2}(\lambda)=\sin\pi\lambda -V_1(\lambda) -\int_{a-\pi}^{\pi-a} (p_{11}-p_{22})(x)\cos\lambda x\,dx + \int_{a-\pi}^{\pi-a}
(p_{12}+p_{21})(x)\sin\lambda x\,dx,
\end{equation}
where
\begin{equation}\label{2.8}
V_1(\lambda)=\omega_1\cos\lambda(\pi-a) +\omega_2\sin\lambda(\pi-a), \quad V_2(\lambda)=\omega_1\sin\lambda(\pi-a)
-\omega_2\cos\lambda(\pi-a),
\end{equation}
\begin{equation}\label{2.9}
\omega_1=\frac12\int_a^\pi(q_{11}+q_{22})(x)\,dx, \quad \omega_2=\frac12\int_a^\pi(q_{12}-q_{21})(x)\,dx, \quad p_{\nu j}(x)=\frac14 q_{\nu
j}\Big(\frac{\pi+a-x}2\Big).
\end{equation}

\medskip
{\bf Remark 2.} Representations (\ref{2.4})--(\ref{2.7}) show that the spectrum of any problem ${\cal B}_{\nu,j}(Q)$ may carry information
only on the functions $q_{11}(x)-q_{22}(x),$ $q_{12}(x)+q_{21}(x)$ and only on the mean values of the functions $q_{11}(x)+q_{22}(x)$ and
$q_{12}(x)-q_{21}(x),$ which justifies our assumption (\ref{1.4}).

Moreover, it can be easily seen that, from the informational point of view, it is sufficient to specify the spectra of any pair of problems
with one common boundary condition, e.g., ${\cal B}_{1,1}(Q)$ and ${\cal B}_{1,2}(Q).$

\medskip
From now on, we apply the assumption (\ref{1.4}). Then, according to (\ref{2.4}), (\ref{2.5}) and (\ref{2.8}), (\ref{2.9}), the
characteristic functions $\Delta_1(\lambda):=\Delta_{1,1}(\lambda)$ and $\Delta_2(\lambda):=\Delta_{1,2}(\lambda)$ of the problems ${\cal
B}_{1,1}(Q)$ and ${\cal B}_{1,2}(Q),$ respectively, take the forms
\begin{equation}\label{2.10}
\Delta_1(\lambda)=-\sin\pi\lambda +\int_{a-\pi}^{\pi-a} w_1(x)\exp(i\lambda x)\,dx,
\end{equation}
\begin{equation}\label{2.11}
\Delta_2(\lambda)=\cos\pi\lambda  +\int_{a-\pi}^{\pi-a} w_2(x)\exp(i\lambda x)\,dx,
\end{equation}
where
\begin{equation}\label{2.12}
w_1(x)=-\frac14(q+ip)\Big(\frac{\pi+a-x}2\Big)-\frac14(q-ip)\Big(\frac{\pi+a+x}2\Big),
\end{equation}
\begin{equation}\label{2.13}
w_2(x)=\frac14(iq-p)\Big(\frac{\pi+a-x}2\Big)-\frac14(iq+p)\Big(\frac{\pi+a+x}2\Big).
\end{equation}

By the standard approach involving Rouch\'e's theorem, one can prove the following assertion, which can also be obtained as a direct
corollary from general Theorem~4 in \cite{But22}.

\begin{lemma}\label{lem1}
For $j=1,2,$ the function $\Delta_j(\lambda)$ has infinitely many zeros $\lambda_{n,j},$ $n\in{\mathbb Z},$ of the form~(\ref{1.3}).
\end{lemma}

We note that Lemma~\ref{lem1} immediately implies the assertion of Theorem~\ref{th1}. The following similar assertion can analogously be
obtained by using (\ref{2.6}) and (\ref{2.7}).

\begin{corollary}\label{th1} Let condition (\ref{1.4}) be fulfilled. For $j\in\{1,2\},$ the problem ${\cal B}_{2,j}(Q)$ has infinitely
many eigenvalues $\theta_{n,j},$ $n\in{\mathbb Z},$ of the form
$$
\theta_{n,j}=n+\frac{j}2+\kappa_{n,j}, \quad \{\kappa_{n,j}\}\in l_2.
$$
\end{corollary}

\section{Related entire functions}\label{sec3}

\medskip
Consider the entire functions
\begin{equation}\label{3.1}
\Delta_1(\lambda)=-\sin\pi\lambda +\int_{-\pi}^{\pi} w_1(x)\exp(i\lambda x)\,dx, \quad w_1(x)\in L_2(-\pi,\pi),
\end{equation}
\begin{equation}\label{3.2}
\Delta_2(\lambda)=\cos\pi\lambda  +\int_{-\pi}^{\pi} w_2(x)\exp(i\lambda x)\,dx, \quad w_2(x)\in L_2(-\pi,\pi).
\end{equation}
which differ from the characteristic functions (\ref{2.10}) and (\ref{2.11}) by the maximal possible support of the functions $w_1(x)$ and
$w_2(x)$ under the integral. In the next section, we will actually show that it is an only difference, i.e. any functions of the form
(\ref{2.10}) and (\ref{2.11}) with square-integrable $w_1(x)$ and $w_2(x)$ are the characteristic functions of some pair of problems ${\cal
B}_{1,1}(Q)$ and ${\cal B}_{1,2}(Q),$ respectively.

In this section, we first collect necessary properties of the functions (\ref{3.1}) and (\ref{3.2}). It is convenient to obtain them as
direct corollaries of the corresponding general assertions in \cite{But22}.

\begin{lemma}\label{lem2}
The functions $\Delta_1(\lambda)$ and $\Delta_2(\lambda)$ are uniquely determined by specifying their zeros. Moreover, representations
(\ref{1.8}) hold.
\end{lemma}

\begin{proof}
For $j=1,2,$ Theorem~5 in \cite{But22} gives the representation
\begin{equation}\label{3.3}
\Delta_j(\lambda)=\alpha_j\exp(\beta_j\lambda)\prod_{n\in{\mathbb Z}}\frac{\lambda_{n,j}-\lambda}{\mu_{n,j}}
\exp\Big(\frac\lambda{\mu_{n,j}}\Big), \quad \mu_{n,1}=\left\{\begin{array}{rl} n, & n\ne0,\\ -1, & n=0, \end{array}\right. \quad
\mu_{n,2}=n-\frac12,
\end{equation}
where $\beta_j=2-j+\gamma_j$ and
$$
\alpha_1=-\lim_{\lambda\to0}\frac{\sin\lambda\pi}\lambda =-\pi, \quad \alpha_2=\lim_{\lambda\to0}\cos\lambda\pi =1,
$$
$$
\gamma_1=\lim_{\lambda\to0}\frac{d}{d\lambda}\ln\Big(-\frac{\sin\lambda\pi}\lambda\Big) =0, \quad
\gamma_2=\lim_{\lambda\to0}\frac{d}{d\lambda}\ln\cos\lambda\pi =0.
$$
Hence, formula (\ref{3.3}) for $j=1$ and $j=2$ takes the forms as in (\ref{1.8}), respectively.
\end{proof}

According to the proof of Lemma~\ref{lem2}, the following assertion is a corollary from Theorem~6 in \cite{But22}.

\begin{lemma}\label{lem3}
For any complex sequences $\{\lambda_{n,1}\}_{n\in{\mathbb Z}}$ and $\{\lambda_{n,2}\}_{n\in{\mathbb Z}}$ of the form (\ref{1.3}), the
functions $\Delta_1(\lambda)$ and $\Delta_2(\lambda)$ constructed by the formulae in (\ref{1.8}) have the forms (\ref{3.1}) and (\ref{3.2}),
respectively.
\end{lemma}

Along with the functions (\ref{3.1}) and (\ref{3.2}), we consider other functions $\tilde\Delta_1(\lambda)$ and $\tilde\Delta_2(\lambda)$ of
the same forms but with different functions $\tilde w_1(x)$ and $\tilde w_2(x)$ under the integrals, respectively. If a certain
symbol~$\alpha$ denotes an object related to the function $\Delta_j(\lambda),$ then this symbol with tilde $\tilde\alpha$ will denote the
analogous object related to the function $\tilde\Delta_j(\lambda).$ The following lemma immediately follows from Theorem~7 in~\cite{But22}.

\begin{lemma}\label{lem4}
For $j\in\{1,2\}$ and for any fixed $r>0,$ there exists $C_r>0$ such that the estimate
$$
\|w_j-\tilde w_j\|_{L_2(-\pi,\pi)}\le C_r \|\{\lambda_{n,j}-\tilde\lambda_{n,j}\}_{n\in{\mathbb Z}}\|_{l_2}
$$
is fulfilled as soon as $\|\{\lambda_{n,j}-n+(j-1)/2\}_{n\in{\mathbb Z}}\|_{l_2}\le r$ and $\|\{\tilde\lambda_{n,j}-n+(j-1)/2\}_{n\in{\mathbb
Z}}\|_{l_2}\le r.$
\end{lemma}

Further, let $\{n_k\}_{k\in{\mathbb Z}}$ be an increasing sequence of integer numbers, and let $\{\varkappa_k\}_{k\in{\mathbb Z}}$ be a
complex sequence in $l_2.$ Put $z_k:=n_k+\varkappa_k,$ $k\in{\mathbb Z}.$ Assume for convenience that multiple $z_k$ are neighboring, i.e.
$$
z_k=z_{k+1}=\ldots=z_{k+m_k-1},
$$
where $m_k$ is the multiplicity of $z_k$ in the sequence $\{z_k\}_{k\in{\mathbb Z}}.$ Denote ${\cal S}:=\{k:z_k\ne z_{k-1}, k\in{\mathbb
Z}\},$ and put
$$
e_{k+\nu}(x):=x^\nu\exp(iz_kx), \quad k\in{\cal S}, \quad \nu=\overline{0,m_k-1}.
$$

For proving Theorems~\ref{th2}--\ref{th4}, we will also need the following auxiliary assertion.

\begin{lemma}\label{lem5}
The system $\{e_k(x)\}_{k\in{\mathbb Z}}$ is complete (a Riesz basis) in $L_2(-b,b)$ if and only if so is the system
$\{\exp(in_kx)\}_{k\in{\mathbb Z}}.$
\end{lemma}

\begin{proof}
Let us first prove that the systems $\{e_k(x)\}_{k\in{\mathbb Z}}$ and $\{\exp(in_kx)\}_{k\in{\mathbb Z}}$ can be complete in $L_2(-b,b)$
only simultaneously. Let $\{\exp(in_kx)\}_{k\in{\mathbb Z}}$ be incomplete in $L_2(-b,b).$ Then there exists a nonzero entire function
$u(\lambda)$ of the form
$$
u(\lambda)=\int_{-b}^bf(x)\exp(i\lambda x)\,dx, \quad f(x)\in L_2(-b,b),
$$
whose zeros include the sequence $\{n_k\}_{k\in{\mathbb Z}}.$ Consider the meromorphic function
$$
F(\lambda):=\prod_{k\in{\mathbb Z}}\frac{z_k-\lambda}{n_k-\lambda}.
$$
Then the function $v(\lambda):=F(\lambda)u(\lambda)$ has a similar form:
\begin{equation}\label{3.5}
v(\lambda)=\int_{-b}^bg(x)\exp(i\lambda x)\,dx, \quad g(x)\in L_2(-b,b).
\end{equation}
Indeed, the estimate $|F(\lambda)|\le C_\delta$ holds in $\{\lambda:|\lambda- n_k|\ge\delta,\,k\in{\mathbb Z}\}$ for each fixed $\delta>0$
(see, e.g., the proof of Lemma~2 in \cite{But22}). Obviously, the function $v(\lambda),$ after removing singularities, is entire in~$\lambda$
and, by virtue of the maximum modulus principle, its exponential type does not exceed $b.$ Moreover, the estimate $|v(\lambda)|\le
C_\delta|u(\lambda)|$ holds, in particular, on the horizontal line ${\rm Im}\lambda=\delta.$ Hence, $v(\lambda)\in
L_2(-\infty+i\delta,+\infty+i\delta),$ and (\ref{3.5}) follows from the Paley--Wiener theorem.

Since the function $v(\lambda)$ is not identically zero and its zeros include the sequence $\{z_k\}_{k\in{\mathbb Z}},$ the system
$\{e_k(x)\}_{k\in{\mathbb Z}}$ is not complete in $L_2(-b,b)$ too. Conversely, assuming the incompleteness of $\{e_k(x)\}_{k\in{\mathbb Z}}$
one can analogously show that $\{\exp(in_kx)\}_{k\in{\mathbb Z}}$ is not complete.

Finally, we note that for proving the simultaneous Riesz basisness of the systems $\{e_k(x)\}_{k\in{\mathbb Z}}$ and
$\{\exp(in_kx)\}_{k\in{\mathbb Z}},$ it remains to establish their quadratical closeness, i.e. the inequality
$$
\sum_{k\in{\mathbb Z}}\int_{-b}^b|e_k(x)-\exp(in_kx)|^2\,dx<\infty,
$$
which, in turn, is obvious.
\end{proof}

\section{Solution of the inverse problem}\label{sec4}

\medskip
Before proceeding directly to the proof of Theorems~\ref{th2}--\ref{th6}, we fulfil some preparatory work. For $j=1,2,$ put
$\mu_{k,j}:=\lambda_{n_{k,j},j}$ and consider the set ${\cal S}_j:=\{k:\mu_{k,j}\ne \mu_{k-1,j}, k\in{\mathbb Z}\},$ where, as in the
preceding section, we assume without loss of generality that
$$
\mu_{k,j}=\mu_{k+1,j}=\ldots=\mu_{k+m_{k,j}-1,j},
$$
where $m_{k,j}$ is the multiplicity of $\mu_{k,j}$ in the sequence $\{\mu_{k,j}\}_{k\in{\mathbb Z}}.$ Define
\begin{equation}\label{4.0}
e_{k+\nu,j}(x):=(ix)^\nu\exp(i\mu_{k,j}x), \quad k\in{\cal S}_j, \quad \nu=\overline{0,m_{k,j}-1}, \quad j=1,2.
\end{equation}

Further, consider relations (\ref{2.12}) and (\ref{2.13}). Obviously, for any functions $p(x)$ and $q(x)$ in $L_2(a,\pi)$ they uniquely
determine some functions $w_1(x)$ and $w_2(x)$ in $L_2(a-\pi,\pi-a).$ The inverse assertion holds too. Indeed, rewrite (\ref{2.12}) and
(\ref{2.13}) in the form
\begin{equation}\label{4.1}
4w_1(x)=-(q+ip)\Big(\frac{\pi+a-x}2\Big)-(q-ip)\Big(\frac{\pi+a+x}2\Big),
\end{equation}
\begin{equation}\label{4.2}
4iw_2(x)=-(q+ip)\Big(\frac{\pi+a-x}2\Big)+(q-ip)\Big(\frac{\pi+a+x}2\Big).
\end{equation}
Summing up  (\ref{4.1}) and (\ref{4.2}) and subtracting one from the other we get
$$
2(w_1+iw_2)(x)=-(q+ip)\Big(\frac{\pi+a-x}2\Big), \quad 2(w_1-iw_2)(-x)=-(q-ip)\Big(\frac{\pi+a-x}2\Big),
$$
Acting in the same way for the obtained relations and changing the variable, we get
\begin{equation}\label{4.3}
q(x)=-(w_1+iw_2)(\pi+a-2x)-(w_1-iw_2)(2x-\pi-a),
\end{equation}
\begin{equation}\label{4.4}
p(x)=(iw_1-w_2)(\pi+a-2x)-(iw_1+w_2)(2x-\pi-a),
\end{equation}
which also imply the estimates
\begin{equation}\label{4.5}
\|p\|_{L_2(a,\pi)},\|q\|_{L_2(a,\pi)}\le\sqrt2\Big(\|w_1\|_{L_2(a-\pi,\pi-a)}+\|w_2\|_{L_2(a-\pi,\pi-a)}\Big).
\end{equation}

Now we are in position to give proofs of the theorems.

\medskip
\noindent{\it Proof of Theorem~\ref{th2}.} For $j=1,2,$ differentiating $\nu=\overline{0,m_{k,j}-1}$ times the functions $\Delta_j(\lambda)$
in (\ref{2.10}) and (\ref{2.11}), where $k\in{\cal S}_j,$ and, afterwards, substituting $\lambda=\mu_{k,j}$ into the obtained derivatives and
using the definition (\ref{4.0}), we arrive at the relations
\begin{equation}\label{4.6}
\beta_{n,j}=\int_{a-\pi}^{\pi-a} w_j(x)e_{n,j}(x)\,dx, \quad n\in{\mathbb Z},
\end{equation}
where
\begin{equation}\label{4.7}
\beta_{k+\nu,j}=f_j^{(\nu)}(\mu_{k,j}), \quad k\in{\cal S}_j, \quad \nu=\overline{0,m_{k,j}-1}, \quad
f_j(\lambda):=\left\{\begin{array}{rl}\sin\pi\lambda, & j=1,\\-\cos\pi\lambda, &j=2.\end{array}\right.
\end{equation}
Note that relations (\ref{4.0}), (\ref{4.6}) and (\ref{4.7}) actually mean that the sequences $\{\lambda_{n_{k,1},1}\}_{k\in{\mathbb Z}}$ and
$\{\lambda_{n_{k,2},2}\}_{k\in{\mathbb Z}}$ are subspectra of the problems ${\cal B}_{1,1}(Q)$ and ${\cal B}_{1,2}(Q),$ respectively.

Let these subspectra uniquely determine $Q(x).$ Assume, to the contrary, that at least one of the systems $\{\exp(in_{k,j}x)\}_{k\in{\mathbb
Z}},$ $j=1,2,$ is not complete in $L_2(a-\pi,\pi-a).$ Then, according to Lemma~\ref{lem5}, the same can be said about the systems
$\{e_{n,j}(x)\}_{n\in {\mathbb Z}},$ $j=1,2.$ Hence, there exists another pair $(\tilde w_1,\tilde w_2)$ of functions in $L_2(a-\pi,\pi-a)$
that is different from $(w_1,w_2),$ and the relations
\begin{equation}\label{4.8}
\beta_{n,j}=\int_{a-\pi}^{\pi-a}\tilde w_j(x)e_{n,j}(x)\,dx, \quad n\in{\mathbb Z}, \quad j=1,2,
\end{equation}
are fulfilled. To these $\tilde w_1(x)$ and $\tilde w_2(x),$ via formulae (\ref{4.3}) and (\ref{4.4}) there correspond another pair of
functions $\tilde p(x)$ and $\tilde q(x)$ such that $(\tilde p,\tilde q)\ne(p,q).$ Then, analogously to the above, relations (\ref{4.0}),
(\ref{4.7}) and (\ref{4.8}) will mean that $\{\lambda_{n_{k,1},1}\}_{k\in{\mathbb Z}}$ and $\{\lambda_{n_{k,2},2}\}_{k\in{\mathbb Z}}$ are
subspectra also of the problems ${\cal B}_{1,1}(\tilde Q)$ and ${\cal B}_{1,2}(\tilde Q),$ respectively, with $\tilde Q\ne Q.$ The obtained
contradiction proves the necessity.

For the sufficiency, note that, by virtue of Lemma~\ref{lem5}, each system $\{e_{n,j}(x)\}_{n\in {\mathbb Z}},$ $j=1,2,$ is complete in
$L_2(a-\pi,\pi-a).$ Thus, according to (\ref{4.6}) and (\ref{4.7}), the functions $w_1(x)$ and $w_2(x)$ are uniquely determined by the
subspectra. Then relations (\ref{4.3}) and (\ref{4.4}) imply uniqueness of $p(x)$ and $q(x).$  $\hfill\Box$

\medskip
\noindent{\it Proof of Theorem~\ref{th3}.} Let the systems $\{\exp(in_{k,1}x)\}_{k\in{\mathbb Z}}$ and  $\{\exp(in_{k,2}x)\}_{k\in{\mathbb
Z}}$ be Riesz bases in $L_2(a-\pi,\pi-a).$ Then, by Lemma~\ref{lem5}, so are the systems $\{e_{n,1}(x)\}_{n\in {\mathbb Z}}$ and
$\{e_{n,2}(x)\}_{n\in {\mathbb Z}}$ defined by~(\ref{4.0}). According to (\ref{1.6}) and (\ref{4.7}), for sufficiently large $|k|,$ we have
$$
\beta_{k,1}=\sin\pi(n_{k,1}+\varkappa_{k,1})=(-1)^{n_{k,1}}\sin\pi\varkappa_{k,1},
$$
$$
\beta_{k,2}=-\cos\pi\Big(n_{k,2}-\frac12+\varkappa_{k,2}\Big)=-(-1)^{n_{k,2}}\sin\pi\varkappa_{k,2}.
$$
Hence, $\{\beta_{k,j}\}_{k\in\mathbb Z}\in l_2.$ Thus, for each $j=1,2,$ there exists a unique function $w_j(x)\in L_2(a-\pi,\pi-a)$ obeying
(\ref{4.6}). Let $q(x)$ and $p(x)$ be constructed by formulae (\ref{4.3}) and (\ref{4.4}) with these $w_1(x)$ and $w_2(x),$ which, in turn,
obviously appear in the representations (\ref{2.10}) and (\ref{2.11}) for the characteristic functions of the corresponding problems ${\cal
B}_{1,1}(Q)$ and ${\cal B}_{1,2}(Q).$ Then relations (\ref{4.0}), (\ref{4.6}) and (\ref{4.7}) mean that the given sequences
$\{\mu_{k,1}\}_{k\in{\mathbb Z}}$ and  $\{\mu_{k,2}\}_{k\in{\mathbb Z}}$ are their subspectra, respectively, which finishes the proof.
$\hfill\Box$

\medskip
The proof of Theorem~3 is constructive and gives the following algorithm for solving Inverse Problem~2 under the hypotheses of
Theorem~\ref{th3}

\medskip
{\bf Algorithm 1.} {\it Let the subspectra $\{\lambda_{n_k,j}\}_{k\in {\mathbb Z}},\,j=0,1,$ be given. Then:
\\
(i) For $j=1,2,$ construct $\{e_{n,j}(x)\}_{n\in {\mathbb Z}}$ by (\ref{4.0}) and $\{\beta_{n,j}\}_{n\in {\mathbb Z}}$ by (\ref{4.7}) using
$\mu_{k,j}=\lambda_{n_k,j}$ for $k\in {\mathbb Z};$
\\
(ii) Find the functions $w_1(x)$ and $w_2(x)$ by the formulae
$$
w_1(x)=\sum_{n=-\infty}^\infty \beta_{n,1}e_{n,1}^*(x), \quad w_2(x)=\sum_{n=-\infty}^\infty \beta_{n,2}e_{n,2}^*(x),
$$
where $\{e_{n,1}^*(x)\}_{n\in {\mathbb Z}}$ and $\{e_{n,2}^*(x)\}_{n\in {\mathbb Z}}$ are the Riesz bases that are biorthogonal to the bases
$\{\overline{e_{n,1}(x)}\}_{n\in {\mathbb Z}}$ and $\{\overline{e_{n,2}(x)}\}_{n\in {\mathbb Z}},$ respectively;
\\
(iii) Construct the functions $q(x)$ and $p(x)$ by formulae (\ref{4.3}) and (\ref{4.4}).}

\medskip
\noindent{\it Proof of Theorem~\ref{th4}.} The necessity follows from the asymptotics (\ref{1.3}) along with the definition of $m$-th
subspectra. For the sufficiency, it is enough to note that the system $\{\exp(imkx)\}_{k\in {\mathbb Z}}$ is an orthogonal basis in
$L_2(-\pi/m,\pi/m),$ and to apply Theorem~\ref{th3}. $\hfill\Box$

\medskip
\noindent{\it Proof of Theorem~\ref{th5}.} By necessity, the asymptotics (\ref{1.3}) was already established in Theorem~\ref{th1}. According
to Lemma~2, the characteristic functions $\Delta_1(\lambda)$ and $\Delta_2(\lambda)$ have the representations (\ref{1.8}). Thus, condition
(ii) easily follows from representations (\ref{2.10}) and (\ref{2.11}). For the sufficiency, we construct the functions $\Delta_1(\lambda)$
and $\Delta_2(\lambda)$ by (\ref{1.8}) using the given sequences $\{\lambda_{n,1}\}_{n\in{\mathbb Z}}$ and $\{\lambda_{n,2}\}_{n\in{\mathbb
Z}}.$  By Lemma~\ref{lem3}, these functions have the forms (\ref{3.1}) and (\ref{3.2}) with some functions $w_1(x),w_2(x)\in L_2(-\pi,\pi),$
respectively. Further, condition (ii) along with the Paley--Wiener theorem implies $w_j(x)=0$ a.e.~on $(-\pi,a-\pi)\cup(\pi-a,\pi),$ i.e.
representations (\ref{2.10}) and (\ref{2.11}) hold. Construct the functions $p(x)$ and $q(x)$ by formulae (\ref{4.3}) and (\ref{4.4}) and
consider the corresponding problems ${\cal B}_{1,1}(Q)$ and ${\cal B}_{1,2}(Q).$ Then, as in Section~\ref{sec2}, one can show that
$\Delta_1(\lambda)$ and $\Delta_2(\lambda)$ are their characteristic functions, respectively. $\hfill\Box$

\medskip
\noindent{\it Proof of Theorem~\ref{th6}.} Applying Lemma~\ref{lem4} to the characteristic functions (\ref{2.10}) and (\ref{2.11}), we get
$$
\|w_j-\tilde w_j\|_{L_2(a-\pi,\pi-a)}\le C_r \|\{\lambda_{n,j}-\tilde\lambda_{n,j}\}_{n\in{\mathbb Z}}\|_{l_2}, \quad j=1,2,
$$
as soon as the conditions of Theorem~\ref{th6} are met. Hence, the estimate (\ref{1.9}) follows from the estimates~(\ref{4.5}) along with the
linearity of relations (\ref{4.3}) and (\ref{4.4}). $\hfill\Box$

\medskip
Finally, note that Algorithm~1 can be used also for solving Inverse Problem~1 by putting $n_{k,j}:=k$ for $j=1,2.$ However, it can be
simplified in the following way because the complete spectra are known, and one can use an orthogonal basis.

\medskip
{\bf Algorithm 2.} {\it Let the complete spectra $\{\lambda_{n,j}\}_{k\in {\mathbb Z}},\,j=0,1,$ be given. Then:
\\
(i) Construct the functions $\Delta_1(\lambda)$ and $\Delta_2(\lambda)$ by formulae (\ref{1.8});
\\
(ii) In accordance with (\ref{2.10}) and (\ref{2.11}), find the functions $w_1(x)$ and $w_2(x)$ by the formulae
$$
w_1(x)=\frac1{2\pi}\sum_{n=-\infty}^\infty \Delta_1(n)\exp(-inx), \quad w_2(x)=\frac1{2\pi}\sum_{n=-\infty}^\infty
\Big(\Delta_2(n)-(-1)^n\Big)\exp(-inx);
$$
\\
(iii) Construct the functions $q(x)$ and $p(x)$ by formulae (\ref{4.3}) and (\ref{4.4}).}

\medskip
\begin{acknowledgments}
This research was supported by grant of the Russian Science Foundation No. 22-21-00509, https://rscf.ru/project/22-21-00509/.
\end{acknowledgments}

\end{document}